\documentclass[10pt]{amsart}
\usepackage[utf8]{inputenc}
\usepackage{amssymb}
\usepackage{amsfonts, color}
\usepackage{graphicx}
\usepackage{amsmath,amsthm}%
\usepackage[normalem]{ulem}
\providecommand{\U}[1]{\protect\rule{.1in}{.1in}}
\newtheorem{theorem}{Theorem}[section]

\newtheorem{corollary}[theorem]{Corollary}

\newtheorem{lemma}[theorem]{Lemma}

\newtheorem{proposition}[theorem]{Proposition}

\newtheorem{remark}[theorem]{Remark}

\newcommand{\R}{\mathbb{R}}
\newcommand{\rr}{\mathbb{R}}
\renewcommand{\H}{\mathbb{H}}
\newcommand{\N}{\mathbb{N}}
\newcommand{\dive}{{\rm div\, }}
\newcommand{\grad}{{\rm grad\, }}
\newcommand{\arctanh}{{\rm arctanh\, }}

\newcommand{\Om}{\Omega_m}

\usepackage{cancel}
\usepackage{enumitem}  
\makeatletter
\providecommand{\@LN}[2]{}
\makeatother

\begin{document}
	\title[CMC graphs in $\H^{n+1}$ defined in exterior domains]{Constant mean curvature graphs in $\H^{n+1}$ defined in exterior domains}
	\author{Patricia Klaser,\, Adilson Nunes \and Jaime Ripoll }
	
	\begin{abstract}
		Given $H\in [0,1)$ and given a $C^0$ exterior domain $\Omega$ in a $H-$hyper\-sphere of $\H^{n+1},$ the existence of hyperbolic Killing graphs of CMC $H$ defined in $\overline{\Omega}$ with boundary $ \partial \Omega $ included in the $H-$hypersphere is obtained.
	\end{abstract}
	
\maketitle

{\bf MSC Class:} Primary 53A10. Secondary 53C42.	
	

	\section{Introduction}

In this paper we investigate the existence of solutions to the exterior
Dirichlet problem (EDP) for the constant mean curvature (CMC) equation, with vanishing boundary data, in the hyperbolic space $\H^{n+1},$ for hyperbolic Killing graphs (that is, for hypersurfaces which can be represented as Euclidean
radial graphs centered at the origin in the half space model of $\H^{n+1}$). We
recall that a Dirichlet problem is called exterior when it is defined on a domain which complement is compact.

The exterior Dirichlet problem for the minimal surface equation was first
considered in the Euclidean space, seemingly by Nitsche \cite{Ni}, who proved
that if $u$ is a solution to an exterior Dirichlet problem over a domain in
the Euclidean plane, then it has at most linear growth and its graph has a
well defined Gauss map at infinity. This problem has also been investigated,
considering diferent aspects (existence, uniqueness, behaviour at infinity and
foliations of open subsets of the ambient space by solutions) by P. Collin and
R. Krust \cite{CK}, E. Kuwert \cite{Ku}, Kutev and Tomi \cite{KT}, the third author and Tomi  \cite{RT}, and
more recently by Aiolfi, Bustos and the third author \cite{ABR} all in the Euclidean space $\rr^n$
for minimal hypersurfaces. In \cite{ABR}, a brief survey on previous results about the EDP
for the minimal hypersurface equation in $\rr^n$ is presented.

In \cite{GS}, Guan and Spruck proved a result that contrasts with Bernstein theorem in $\rr^3$: the existence of entire non trivial solutions for CMC
$0\leq H<1$ hyperbolic Killing graphs in $\H^{n+1}$; moreover, the solutions can assume
any prescribed boundary data at infinity.

Motivated by the Euclidean history and by the pioneer work of Guan-
Spruck on CMC entire radial graphs in the hyperbolic space, we investigate here
the EDP for hyperbolic graphs in the hyperbolic space  $\H^{n+1}$. We first observe that
these graphs can be defined in a coordinate free form as hyperbolic Killing graphs
as follows.

Let $X$ be a hyperbolic Killing vector field of $\H^{n+1}$ that is, the integral
curves of $X$ are hypercycles orthogonal to a totally geodesic hypersurface,
say $\H^{n}$ of $\H^{n+1}$. Denote by $\{\varphi_t\}_{t\in \rr},$ $\varphi_0 = Id_{\H^{n+1}},$ the flow of $X$, a one-parameter subgroup of isometries of $\H^{n+1}.$ Given a domain
$\Omega\subset \H^n,$ the
$X-$(Killing) graph ${\rm Gr}(u)$ of a function $u$ defined in $\Omega$ is
$${\rm Gr}(u) = \left\{\varphi_{u(x)} (x)\, |\, x\in \Omega\right\}.$$

	\begin{theorem}\label{teo-H=0}
	Let $X$ be a hyperbolic Killing field of $\H^{n+1}$ orthogonal to a totally geodesic hypersurface $\H^n.$
		Let $\Omega \subset \mathbb{H}^n$ be a $C^0$ exterior domain, that is, $\H^n\backslash \Omega$ is compact. Then, given $s\ge 0,$ there is $u\in C^\infty(\Omega)\cap C^0(\overline{\Omega}),$ such that:
		\begin{enumerate}[label=(\roman*)]
			\item the $X-$graph of $u$ is a minimal hypersurface of $\H^{n+1},$
			\item $u\ge 0$ in $\Omega$ and $u=0$ on $\partial\Omega,$ \label{it-2}
			\item $\displaystyle \limsup_{x\to \partial\Omega} |\grad u|=s,$ \label{it-3}
			\item $\displaystyle \sup_{\Omega} |u|\leq \int_{0}^{+\infty}\frac{e^{-2t}}{\cosh(t)\sqrt{1-e^{-4t}}}dt=B(0)<0.712.$
		\end{enumerate}	
	\end{theorem}

The exact value of $B(0)$ is given in \eqref{eq-c0}.

A question that comes up is if one can drop out conditions \ref{it-2} and \ref{it-3} in
Theorem \ref{teo-H=0} and prescribe, instead, the values of $u$ at $\partial \Omega$ and at $\partial_\infty\H^n.$ As to
$\partial\Omega$ we recall that, in the Euclidean space, R. Osserman proved the existence
of a boundary data on the disk $D\subset \rr^2$ for which the EDP in $\rr^2\backslash D$ for the minimal surface equation has no bounded solution \cite{O}. Although not having a similar
example here, we do believe that, due to the non mean convexity of $\Omega,$ such
phenomenon also holds for the EDP of a disk in $\H^2$ (irrespective of the
values of the solution at infinity).

As to $\partial_\infty\H^n,$ considering the results of Guan and Spruck \cite{GS}, and that
the continuous extension to $\partial_\infty\H^n$ of a prospective solution depends on the existence of barriers at infinity, which are local in nature, we could expect to be able to prescribe a boundary data at infinity for the EDP.
However, this is far from being true. Indeed, the zero boundary data of the solution
at $\partial\Omega$ imposes a strong restriction on its height. For instance, there are no minimal hyperbolic Killing graphs on the complement of the unit disk $D$ centered at the origin of $\H^n$ which vanish at $\partial D$ and which are a constant $C$ at the asymptotic boundary of $\H^n$ if $C \ge B(0)$ (see Remark \ref{rem-naoexist}).

This is the reason we prescribed, instead, the extremal inclination of
the solution at the boundary of the domain, as in \cite{RT,ESR}. Nevertheless, we
highlight the open problem of understanding the set of admissible asymptotic
boundary data, i. e., the data at infinity for which the EDP has at least one solution.

As one can observe in the proof of our main results, the idea of prescribing
the extremal inclination of the solution at the boundary of the domain has
a physical motivation: One can imagine two wires in $\overline{\H}^3$, one fixed at $\partial\Omega \subset \H^2$ and the other at the asymptotic boundary of $\H^2.$ In this
configuration, there is a minimal surface, namely the domain $\Omega,$ bounded
by these wires. We can then move upwards the out wire until its maximal
inclination at $\partial \Omega$ is reached.

Another problem we deal with, also motivated by the result of Guan and
Spruck in \cite{GS}, is if it is possible to extend Theorem \ref{teo-H=0} to the EDP for constant mean curvature (bigger than or equal to 0 and strictly less than 1) Killing
graphs: The answer is yes once we formulate the EDP on on another but also
natural way.

Let $\H^n$ be a totally geodesic hypersurface of $\H^{n+1}$ and $X$ a hyperbolic Killing vector field orthogonal to it: Given $H\in [0, 1),$ set 
$$E^n_H:=\{p\in \H^{n+1}\,|\, d(p, \H^n ) = -\arctanh(H) \},$$
where $d$ is the Riemannian oriented distance in $\H^{n+1}$ ($d(p, \H^n )> 0$ if $p$ lies in the side of $\H^{n+1}\backslash \H^n$ to which the field $X$ points to). One may see that $E^n_H$ is a totally umbilical hypersurface of constant mean curvature $H$, if oriented in the direction of $X$. With the induced metric of $\H^{n+1},$ $E^n_H$ is isometric to a hyperbolic space with some negative sectional curvature $k\in [-1, 0).$ 

We may introduce the notion of $X-$graphs over $E^n_H$ as before, in such a way that a graph over an exterior domain in $E_H$ corresponds to a graph over an exterior domain in $\H^n,$ and vice-versa. Then we prove Theorem \ref{teoHneq0} which summarizes as Theorem \ref{teoHneq0-intro} and is illustrated in Figure \ref{fig:dominioEH}.

\begin{figure}[h!]
			\centering
			\includegraphics[width=0.6\linewidth]{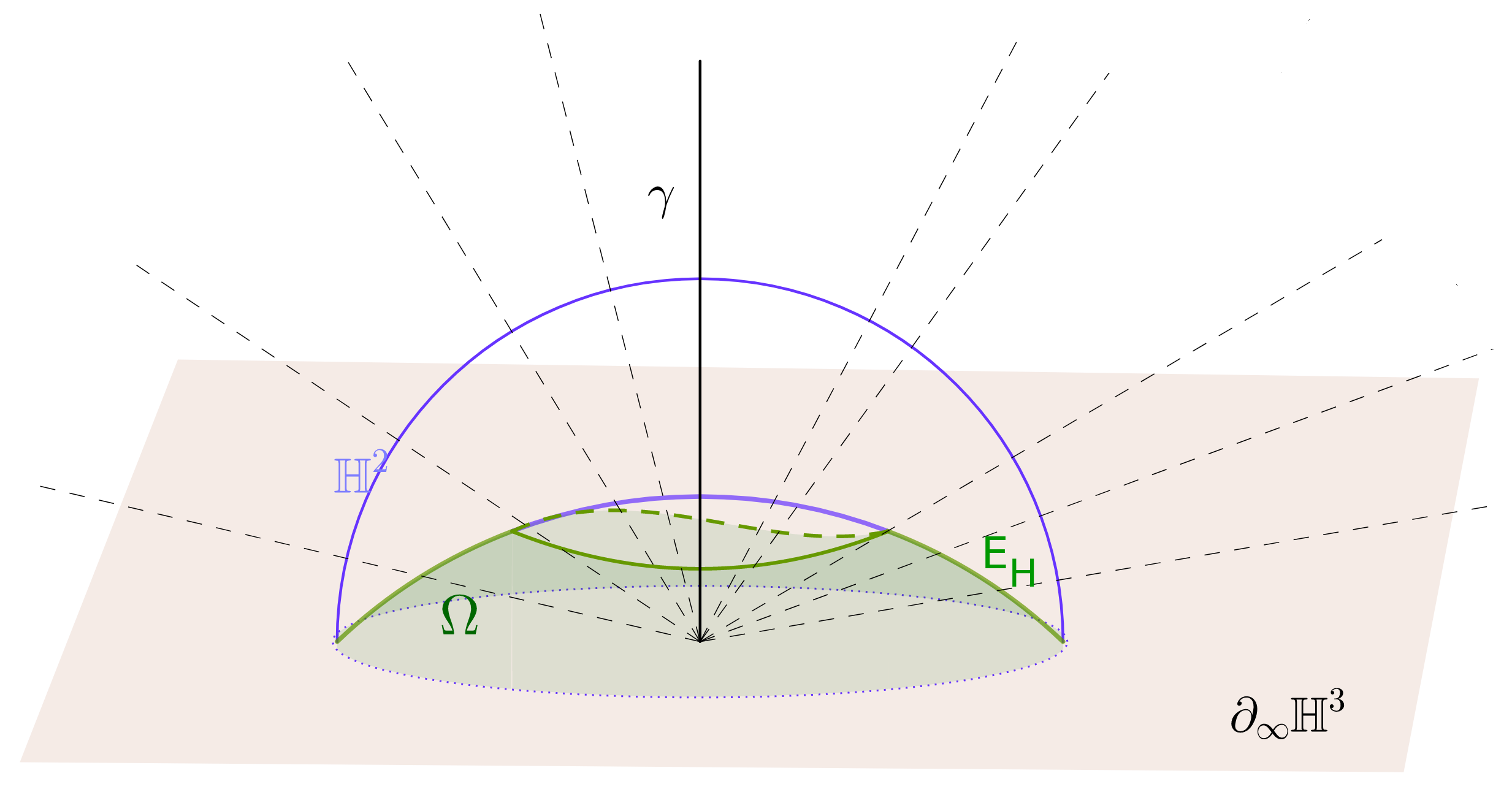}
			\caption{Domain $\Omega$ in $E_H$}
			\label{fig:dominioEH}
		\end{figure}

	\begin{theorem}\label{teoHneq0-intro}
		For $H\in (0,1),$ let $\Omega \subset E^n_H$ be a $C^0$ exterior domain. Then, for any $s\ge 0$ there is an $X-$graph of CMC $H$ given by a positive bounded function over $\Omega,$ with vanishing boundary data and $$\limsup_{x\to \partial\Omega} |\grad u|=s.$$
	\end{theorem}

In \cite{Kr}, \cite{Ku}, \cite{RT} and \cite{ABR} the authors investigate the existence of foliations
associated to the EDP in $\rr^n.$ Foliations of open sets of the ambient spaces
by $H-$surfaces play an important role in the study of complete $H-$surfaces.
A natural continuation of our results is to investigate if some of the
results of these previous papers can be extended to the EDP for hyperbolic
CMC graphs. This is a problem which the authors intend to undertake.

\section{Rotational surfaces}\label{sec-rot}
In the particular case of $\Omega$ being the exterior of a ball centered at the geodesic left invariant by the flux of the hyperbolic Killing field, the EDP reduces to solving an ODE and the solutions correspond to part of the rotational surfaces introduced by do Carmo and Dacjzer in \cite{dCD}. This pieces of rotational surfaces will be used as barriers in the proofs of our main results.

Let $\gamma $ be the  geodesic associated to the hyperbolic Killing field $X$ and $\mathbb{H}^n$ is the totally geodesic surface orthogonal to $X,$ containing $o:=\gamma (0).$ Besides for $x \in \mathbb{H}^n,\; r(x)$ denotes the distance to $o,$ $ D_{\rho}(o)=\{ x \in \mathbb{H}^n\;|\;r(x)<\rho \} $   and $|X(x)|=\cosh (r(x))$ for any $x \in \H^n.$
	
	For instance we may consider the half-space model for $\mathbb{H}^{n+1},$ that is, $\mathbb{R}^{n+1}_{+}$ with the metric $ds^2 = \left(1/x_{n+1}^{2}\right)dx^2$, where $dx$ is the Euclidean metric. Then, if the geodesic $\gamma$ is the oriented $x_{n+1}$ axis $\gamma(t)=(0,e^t),$ then $$\varphi_t (x)=e^{t}x,\, \H^n=\{x_1^2+x_2^2+\dots +x_{n+1}^2=1,\;\;x_{n+1}>0\}\,\text{and } X(p)=p$$  for any $p\in \H^n.$ Furthermore the $X-$graphs are Euclidean radial graphs. We remark that we may assume that $X$ is this hyperbolic Killing field without loss of generality, since $\H^{n+1}$ is a homogeneous space. Nevertheless, we will take this model into account only for the figures.
	
	From Proposition 2.1 of \cite{DR}, a function $u:\Omega \to \R$ has CMC $H$ $X-$graph oriented with normal vector pointing in the direction of $X,$ if it satisfies $\mathcal{M}_H(u)=0$ for 
	\begin{equation}\label{operador}
	\mathcal{M}_H(u):=\dive\left( \frac{\cosh( r) \grad u}{\sqrt{1+\cosh^2(r)|\grad u|^2}}\right) +\frac{\langle\grad u,\sinh(r)\grad r\rangle}{\sqrt{1+\cosh^2(r)|\grad u|^2}}-nH,
	\end{equation}
 where $\dive$ and $\grad$ are the divergent and gradient respectively, in $\mathbb{H}^{n}.$

	\begin{proposition}\label{rad-sim}
		Given $\rho>0$ and $H\in(-1,1),$  there exists a function $ v_{\rho, H}:\mathbb{H}^n\setminus D_{\rho}(o)   \to \mathbb{R}$ such that $ v_{\rho,H}|_{\partial D_{\rho}}=0$ and the $ X-$graph of $ v_{\rho, H} $ has CMC $H.$ Besides, $$\displaystyle{\lim_{x\to \partial D_{\rho},\; x\notin D_{\rho}}|\grad v_{\rho,H}(x)|=+\infty,}$$
		and the $X-$graph of $v_{\rho, H}$ glued with the $X-$graph of  $-v_{\rho, H}$ is a rotational CMC $H$ surface presented in \cite{dCD}.
	\end{proposition}	
	\begin{proof}
	
	For $ u = f \circ r$ in $\mathbb{H}^n\setminus D_{\rho}(o)$, equation $\mathcal{M}_H(u)=0$ becomes an ODE for $f,$ which reads as	$$	g'(r)+g(r)((n-1)\coth(r)+\tanh(r))-nH=0\text{ in } (\rho, +\infty),$$
for $$g(r)=\frac{\cosh(r)f'(r)}{\sqrt{1+\cosh^2(r)f'^2(r)}} .$$
	This is the case since $\Delta r=(n-1) \coth (r) $ and $ |\grad r|=1.$ Besides, the gradient of $v$ being vertical at $\partial D_{\rho}$ means that $f'(\rho)=+\infty$ so that $g(\rho)=1.$
	The solution $g$ is given by $$g(r)=\frac{A(\rho,H)}{\sinh^{n-1}(r) \cosh ( r)}+H\tanh(r), \,r\in [\rho, \infty),$$
	where $$A(\rho,H)=\sinh^{n-1}(\rho) (\cosh (\rho) - H\sinh (\rho) ).$$

	Since the function $x \mapsto \frac{x}{\sqrt{1+x^2}},$ maps one-to-one $(0,+\infty)$ onto $(0,1)$ and has inverse $y \mapsto \frac{y}{\sqrt{1-y^2}},$ for $y\in (0,1),$ we may find the function $f$ by
	
	 \begin{equation}\label{rotsim}
		f(r)=\int_{\rho}^{r}\frac{g(t)}{\cosh(t)\sqrt{1-g(t)^{2}}}dt, \, r>\rho.
		\end{equation}
		
		So defining $ v_{\rho,H}(x)=f(r(x)),$ the result holds. One may prove that $0<g(r)<1$ and that the improper integral above is well defined.
	
	\end{proof}  
	
	Since the function $f'$ varies from $+\infty$ to zero and depends only on $r,$ the EDP for the exterior of any ball centered at $\gamma(0)$ is solvable for any norm of the gradient at the boundary.

	The next result is a consequence of the definition of $v_{\rho,H}.$
	
	\begin{lemma}\label{lem-grdtendezero}
		If $ 0\leq H \leq 1,$ then $ v_{\rho, H}:\H ^n\backslash D_\rho \to \R  $ is an increasing function of the distance to $\partial D_\rho$
   		$$\displaystyle \lim_{r(x) \to \infty}|\grad v_{\rho, H}(x)| =0 \text{ for } H \in [0,1); \text{ and }
\displaystyle \lim_{r(x) \to \infty}|\grad v_{\rho, 1}(x)| >0.$$
  
 \end{lemma}

Therefore the function $v_{\rho, 1}$ is unbounded. Nevertheless, we have

	\begin{lemma}\label{lem-radiais}
		Given $H$ in $[0,1),$ $v_{\rho, H}:\H ^n\backslash D_\rho \to \R $ 
		is bounded above by  
\begin{equation}\label{eq-BH}
		B(H)=\int_{0}^{+\infty}\frac{H+(1-H)e^{-2t}}{\cosh(t)\sqrt{1-(H+(1-H)e^{-2t})^2}}dt.
\end{equation}		
$B$ is an increasing function of $H$ defined in $[0,1),$ satisfying \begin{equation}\label{eq-c0}
B(0)=\frac{\Gamma(1/4)\Gamma(5/4)-\Gamma(3/4)^2}{\sqrt{2\pi}}
\end{equation}
and $\lim_{H\to 1^-} B(H)=+\infty.$
	\end{lemma}

	\begin{proof}
		Let us write $v_{\rho, H}(r(x))=f(r(x))$ for $$f(r)=\int_{0}^{r}\frac{x_{\rho, n}(t)}{\cosh(\rho+t)\sqrt{1-x_{\rho,n}(t)^2}}dt, \text{  where }$$

$$x_{\rho,n}(t)=\left(\frac{\sinh(\rho)}{\sinh(\rho+t)}\right)^{n-1} \frac{ (\cosh (\rho) - H\sinh (\rho))}{\cosh(\rho+t)}+H\tanh(\rho+t).$$

Since $\cosh$ and $a(x)=x/{\sqrt{1-x^2}}$ are increasing functions in $[0,\infty),$

$$f(r)\leq \int_{0}^{r}\frac{\overline{x}(t)}{\cosh(t)\sqrt{1-\overline{x}(t)^2}}dt,$$

for any upper bound $\overline{x}(t)\geq x_{\rho,n}(t).$

Since $\sinh(\rho)/\sinh(\rho+t)<1,$ it holds $x_{\rho,n}(t)\leq x_{\rho, 2}(t).$ Therefore an upper bound for $n=2$ is enough to prove the lemma.

	In this case, we omit the subscript 2 and $$x_{\rho}(t)=\frac{H\cosh(2\rho + 2t)+\sinh(2\rho)-H\cosh(2\rho)}{\sinh(2\rho + 2t)}. $$

Hence $\rho\mapsto x_{\rho}(t)$ is an increasing function of $\rho$ and $$ \displaystyle \lim_{\rho \to \infty} x_{\rho}(t)=H+(1-H)e^{-2t},$$ so that the bound $B(H)$ given in the statement is obtained.
		Since $B'(H)>0$ for $H\in [0,1),$ $B$ is an increasing function. Its value for $H=0$ was computed at www.wolframalpha.com.
	\end{proof}	
	
		The following result follows from Lemma 11 of \cite{DHL}.
	\begin{lemma}\label{lem-gradest}
		Let  $ \Omega $ be a $ C^{2,\alpha} $ bounded open subset of $\H^n$ and $ u \in C^{3}(\Omega)\cap C^1(\overline{\Omega}) $ a solution of $ \mathcal{M}_H(u)=0 $ in $ \Omega $. Assume that $u$ is bounded in $ \Omega $ and that $|\grad u|$ is bounded on $  \partial \Omega $. Then $|\grad u|$ is bounded in $ \Omega $ by a constant that depends only on $ \displaystyle\sup_{\Omega} |u| $
		and $ \displaystyle \sup_{\partial \Omega}|\grad u| .$
	\end{lemma}

	\section{Main results}\label{sec-main}

	Our existence result for the minimal hypersurfaces is inspired on its analogous version in $M^2\times \R$ proved in \cite{ESR}.

	\begin{proof}[Proof of Theorem \ref{teo-H=0}]
		
		We start with the case that $\Omega$ is a $C^\infty$ domain. 
		Let $D_0\subset \H^n$ be the smallest ball centered at $o$ with $\H^n\backslash\Omega\subset \overline{D_0}.$ Let $R_1>R_0$ be such that $$|\grad v_{R_0}(x)|_{\partial D_{R}}<s/2 \text{ for all } R\geq R_1,$$ which exists since $$\displaystyle \lim_{r(x)\to \infty}|\grad v_{R_0}(x)|=0$$ 
  We define the sequence $R_m=R_1+m,\, m\ge 2$ and set $\Omega_m=D_{R_m} \cap \Omega.$
  ~\\
		
		{\bf Claim:} For each $m\in \N ,$ there is a function $u_m\in C^\infty(\overline{\Omega_m}),$ such that $$\mathcal{M}_0(u)=0, \,  u_m=0 \text{ on }\partial\Omega \text{ and }\displaystyle{\sup_{\partial\Omega} |\grad u_m|= s.}$$
		
		Observe that $\Om$ is a connected domain and its boundary is $\partial \Omega\cup \partial D_{R_m}.$ Set $\Gamma_m=\partial D_{R_m}$ and define
 
  \begin{equation*}
		T_m=\left\{t\ge 0\,\big{|}\, \exists\, u_t\in C^\infty (\overline{\Om}) \text{ solution of }PD_m(t)\right\},
		\end{equation*}
where
\begin{equation}\tag{PD$_m(t)$}\label{PDm}
\left\{\begin{array}{ll}
 \mathcal{M}_0(u)=0 &\text{ in } \Omega_m,\\
  u=0 &\text{ on }{\partial\Omega},\\ 
  u =t &\text{ on }{\Gamma_m},\\
 |\grad u_t|\leq s &\text{ on }{\partial\Omega.}
\end{array}\right.
\end{equation}

		We will prove that $h(m)=\sup T_m$ is well defined and belongs to $T_m.$ Then we demonstrate that $u_m,$ the function associated to $h(m),$ satisfies the claim.

		Since $0 \in T_m,$ $T_m$ is not empty. We prove in the sequel that  $\sup T_m \leq B(0),$ for $B(0)$ given by \eqref{eq-c0}. Assume for contradiction that there is $T\in T_m,$ $T>B(0).$ From Lemma \ref{lem-radiais}, there is a function $v_R: \H^n\backslash D_R \to \R$ for each $R\leq R_m,$ that vanishes on $\partial D_R,$ has a minimal $X-$graph which is tangent to the $X-$Killing cylinder that passes through $\partial D_R$. Besides, all these functions are bounded by $B(0).$
		
		\begin{figure}[h!]
			\centering
			\includegraphics[width=0.7\linewidth]{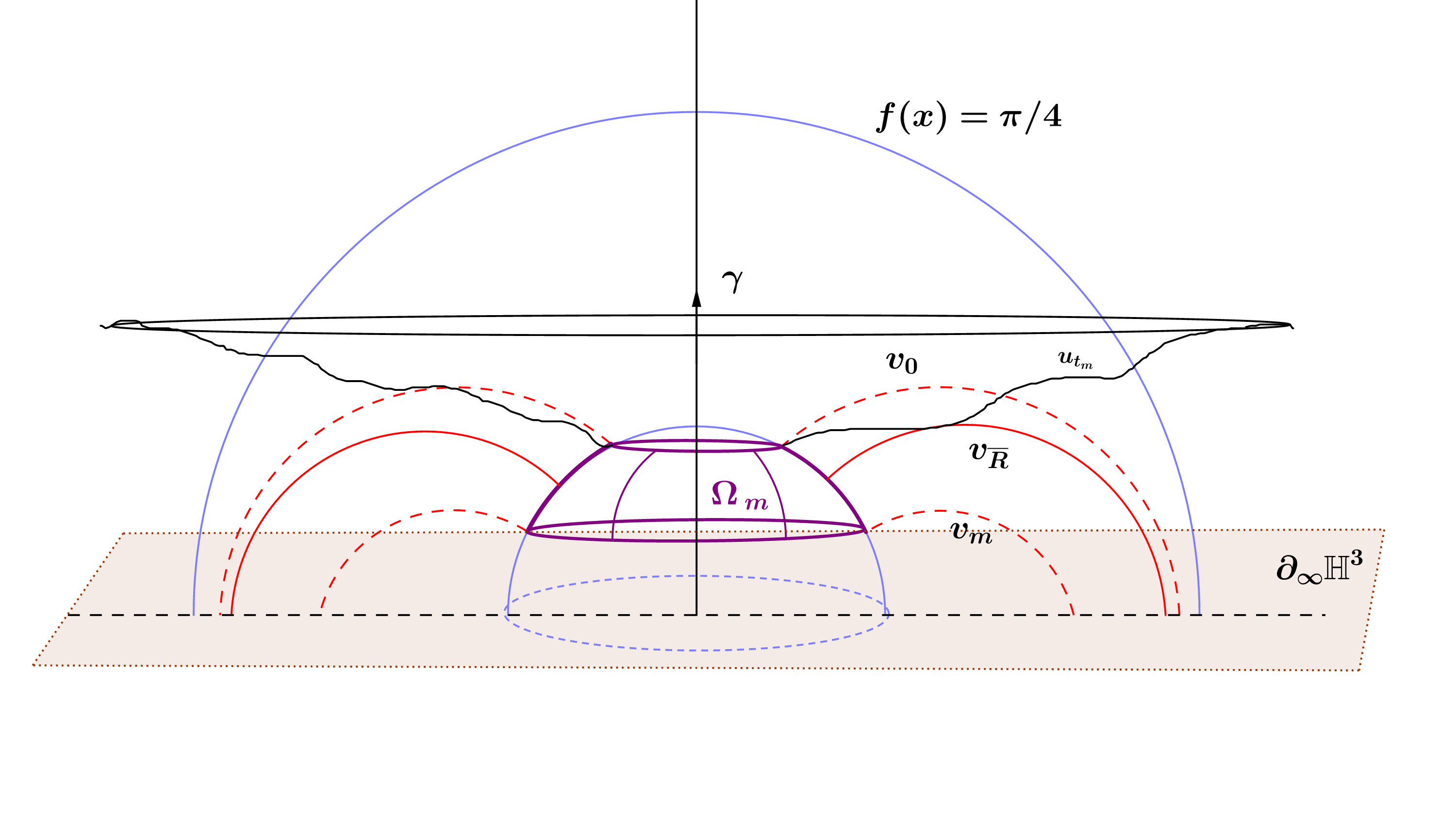}
			\caption{ Maximum principle}
			\label{fig:limitacaoboa}
		\end{figure}

		Notice that for $R=R_m$ the $X-$graphs of $u_T$ and $v_{R_m}$ are disjoint since the intersection of the domains of these functions is $\Gamma_m$ and on $\Gamma_m,$ $$u_T=T>B(0)>0=v_{R_m}.$$ Now consider the family $\{v_R\}_{R\in [R_0,R_m]}$ and define $\bar{R}$ the minimum value of $$\{R\in [R_0,R_m]\,\big{|}\,  {\rm Gr}(v_r)\cap {\rm Gr} (u_T)=\emptyset\, \forall\, r\geq R\},$$ which corresponds to the first contact point of the graph of $u_T$ with a graph of $v_R$ in this family such that the graph of $v_R$ is disjoint from the graph of $u.$ The graph of $v_{\bar{R}}$ touches the graph of $u_T$ at an interior point, leading us to a contradiction with the maximum principle (see Figure \ref{fig:limitacaoboa}). More precisely, the inequality $T>B(0)$ implies that the touching point cannot be in the cylinder over $\Gamma_m,$ and it also cannot be over $\partial\Omega$ because it would contradict $\sup_{\partial \Omega} |\grad u_T|\leq s.$ 
~\\
  
	In order to find the function $u_m=u_{h(m)}$, we prove next that $h(m)$ belongs to $T_m.$ First, we observe that our choice of $R_1$ guarantees that, for all $m,$
		if $t\in T_m,$ $$|\grad u_t|\leq s/2 \text{ on } \Gamma_m.$$ This is again a consequence of the maximum principle, which implies that $$v_{R_0}-(v_{R_0}(R_m)+t)\le u_t\le t \text{ in }\Omega_m.$$ Hence, we have upper (constant $t$) and lower $$v_{R_0}-(v_{R_0}(R_m)+t)$$ barriers to estimate the gradient of $u_t$ on $\Gamma_m$ and therefore for $t\in T_m,$ $$\sup_{\partial\Omega_m} |\grad u_{t}|\leq s.$$
		
		To see that $h(m)\in T_m,$ let $(x_n)$ be a sequence in $T_m$ with $x_n\to h(m).$ Let $u_n=u_{x_n}$ and observe that from the maximum principle  $0\leq u_n \leq h(m)$ in $\Omega_m.$ Besides, $|\grad u_n|$ is bounded by $s$ on $\partial \Omega_m.$ Hence, from Lemma $\ref{lem-gradest},$ there is a constant $C=C(h(m), s, \Omega_m),$ such that $|u_n|_{1,\Omega_m}\leq C.$ Now standard PDE theory implies that the sequence $(u_n)$ has a convergent subsequence in the $C^2$ norm on $\overline{\Omega}$ and, for $w$ being the limit of the sequence, it holds that $w=u_{h(m)}$ so that $h(m)\in T_m.$
		
\		
		
		To prove that $$\sup_{\partial\Omega} |\grad u_{h(m)}|= s,$$ we follow the same idea of \cite{ESR}, which consists in showing that 
		$$\sup_{\partial\Omega} |\grad u_{h(m)}|< s$$ would contradict the fact that $h(m)$ is the supremum of $T_m$. Intuitively, if the strict inequality held, we could raise a little more the boundary data $h(m)$ and still have the gradient bounded by $s.$ 
  
  To be precise, fix any function $\varphi \in C^\infty(\overline{\Omega_m})$ such that $\varphi$ vanishes on $\partial \Omega_m$ and $\varphi\equiv h(m)$ on $\Gamma_m$ and then define an operator $$T:[0,2]\times C^\infty(\overline{\Omega_m}) \to C^\infty(\overline{\Omega_m})\text{ by }$$ $$T(t, w)=\mathcal{M}_0(w+t\varphi).$$
		Notice that $T$ is a $C^1$ operator, $T(1,u_{h(m)}-\varphi)=0$ and $\partial_2T(t,u_{h(m)}-\varphi)$ is an isomorphism. So, the Implicit Function Theorem in Banach spaces states that there is $\delta>0$ and a continuous function $$i:(1-\delta,1+\delta) \to C^\infty(\overline{\Omega_m}),$$ with $i(1)=u_{h(m)}-\varphi$ and such that $T(t, i(t))=0.$ So, there must be a $t_0>1,$ with $|\grad i(t_0)|<s$ on $\partial\Omega.$ This contradicts the definition of $h(m)$ since $t_0h(m)$ would belong to $T_m.$ By taking $u_m=u_{h(m)},$ the claim is demonstrated. 
		
		\
		
		Now consider the sequence of functions $u_m\in C^\infty(\overline{\Omega_m}).$ For any $\Omega_k\subset \Omega$ compact, there is a constant $c(k),$ such that $$|u_m|_{1,\overline{\Omega_k}}\leq c(k) \text{ for all }m> k.$$ Once again, standard PDE theory implies the existence of a  subsequence of $ (u_m)$ converging, in the $ C^2 $ norm on $ \overline{\Omega_k}, $ to a solution $ v_k \in C^{\infty}(\overline{\Omega_k}) $ of $\mathcal{M}_0=0$ in $\Omega_k$ satisfying $v_k=0$ on $\partial \Omega$ and $\sup_{\partial\Omega} |\grad v_k|= s.$ To obtain a solution in $\Omega,$ we iterate this process of taking subsequences: The first subsequence is taken to converge in $\Omega_1,$ then, for $\Omega_2,$ we take a subsequence from the one we already know that converges in $\Omega_1$ and so on. By the diagonal method, we obtain  
		 the existence of a solution $u \geq 0$ of $\mathcal{M}_0=0$ in $\Omega$ satisfying $u=0$ on the boundary and $\sup_{\partial\Omega} |\grad u|= s.$ Since each element of the sequence is bounded by $B(0),$ so is the limit.

		\
		
		For the general case of $\Omega$ being a $C^0$ domain we approximate it by $C^\infty$ domains $(U_j)$ such that $\Omega\subset U_{j+1}\subset U_j$ and $\Omega=\displaystyle{\cap_j}U_j.$ For each $j\in\mathbb{N},$ let $u_j \in C^{\infty}(\overline{U_j})$ be the solution that exists by the case proved above, that is $\mathcal{M}_0(u_j)=0,\;u_j|_{\partial U_j}=0,\; \sup_{\partial \Omega }|\grad u_j|=s\leq |u_j|_1 \leq C(s)$.
		
		From the Arzela-Ascoli Theorem we obtain that $(u_j)$ converges uniformly, on compact subsets of $\overline{\Omega},$ to a function $u_s \in C^{0}(\overline{\Omega})$ such that $u_{s}|_{\partial \overline{\Omega}}=0.$ Besides, regularity theory implies $u_s \in C^{\infty}(\Omega).$ By the diagonal method, we obtain the existence of a subsequence $(u_j)$ converging to a solution $u_s \in C^{\infty}(\Omega )\cap C^{0}(\overline{\Omega})$ of the minimal surface equation.
		
		\end{proof}

	Although this result was inspired in its version for $\mathbb{H}^2\times \R,$ we were not able to construct barriers depending on the distance to the boundary, so we changed the definition of the exhaustion sets $\Omega_m.$ For that, we had to present the rotational barriers from Section \ref{sec-rot} and adapt some steps in the proof. Therefore we were able to avoid the requirement on the existence of an exterior sphere condition Theorem 1 of \cite{ESR}. We remark that the proof of Theorem \ref{teo-H=0} can be applied in any situation where rotational surfaces are known and therefore one can also obtain an improvement of Theorem 1 of \cite{ESR}.

	\begin{theorem}\label{teo-H=0emH2R}
		Let $\Omega \subset \mathbb{H}^2$ be a $C^0$ exterior domain. Then, given $s\ge 0,$ there is $u\in C^\infty(\Omega)\cap C^0(\overline{\Omega}),$ such that:\begin{enumerate}[label=(\roman*)]
			\item the vertical graph of $u$ is a minimal surface in $\mathbb{H}^2\times\R,$
			\item $u\ge 0$ in $\Omega$ and $u=0$ on $\partial\Omega,$
			\item $\displaystyle \limsup_{x\to \partial\Omega} |\grad u|=s,$
			\item $\displaystyle \sup_{\Omega} |u|<\frac{\pi}{2}.$
		\end{enumerate}
	\end{theorem}
	
This result follows the same steps of $\mathbb{H}^3$ if one uses the rotational catenoids 
\begin{equation}\label{eq-def-cat}
\displaystyle\tilde{v}_{\rho, 0}(r) =\int_0^r \frac{\sinh({\rho})}{\sqrt{\sinh^2(t+{\rho})-\sinh^2({\rho})}}dt, \, r>0
\end{equation}	
instead of the functions $v_{\rho, 0}.$ The expression (\ref{eq-def-cat}) can be found in \cite{KST}.

\begin{remark}\label{rem-naoexist} The EDP for $\Omega=\H^n\backslash B_1(0)$ 
$$\left\{\begin{array}{ll}
\mathcal{M}(u)=0 &\text{in }\Omega\\
u =0 &\text{in }\partial\Omega,\\
u =C &\text{in }\partial_\infty\Omega,
\end{array}\right.$$
has no solution for $C\geq B(0)$ for $B(0)$ given by \eqref{eq-c0}.

\end{remark}

\begin{proof}
It is important to observe the the flux $\{\varphi_t\}_{t\in\rr}$ of the hyperbolic Killing field $X$ extends naturally to the asymptotic boundary of $\H^{n+1}.$ Let $$U=\{\varphi_t(x)\,|\, t\in \rr \text{ and } x\in \Omega\}$$ be the Killing cyllinder over $\Omega.$

Assume by contradiction the existence of a complete minimal hypersurface $M$ immersed in $U$ with boundary $\partial M=\partial\Omega$ and $\partial_{\infty}M= \Gamma$ for $$\Gamma=\{\varphi_C(\partial_\infty\H^n)\} \subset\partial_{\infty}\H^{n+1}.$$

Let $S$ be the $X-$graph of $v_{1,0}$ from Proposition \ref{rad-sim}, which is a minimal surface with boundary $\partial\Omega$ and asymptotic boundary $\{\varphi_h(\partial_\infty\H^n)\}$ for some $h<B(0)\leq C.$ Besides it is tangent to $\partial U$ along $\partial\Omega.$

If $M \cap S = \partial\Omega,$ then $S\backslash \partial S$ is contained at the
connected component of $U \backslash M$ which contains $\partial_\infty S.$ In particular, $S$ is at one side of M around $\partial\Omega.$ This contradicts Hopf boundary lemma since $M$ and $S$ have, up to sign, the same unit normal vector along $\partial\Omega.$

If $M \cap S \backslash \partial B$  is nonempty, then $M$ and $S$ have common interior points. Since $M \cap S$ is compact because $h<C,$ we may displace $S$ along the flow
$\{\varphi_t\}_{t<0}$ to get a tangency point between $M$ and $\varphi_{t_0}(S)$ at an interior point, for some
$t_0 < 0$ with $\varphi_{t_0}(S)$ contained in one side of $M,$ a contradiction. 
\end{proof}

\subsection{The case $H\in (0,1)$}	
	
We conclude this manuscript by considering graphs over $E_H,$ which are umbilical hypersurfaces of CMC $H.$	Since for $H \in (0,1),$ $E_H$ is an entire $X-$graph over $\H^n,$ graphs over $E_H$ are also graphs over $\H^n$ and vice versa. Nevertheless, it is natural to assume that our exterior domains are subsets of $E_H.$ This is the case, since the geometric idea of our proof is to consider the exterior domain $\Omega$ as the graph (of the zero function) and lift the boundary data in order to find non trivial solutions.

It is also important to mention that the EDP for $H\neq 0 $ has already been considered in the ambient space $\H^2\times \rr$ in \cite{CS}. An existence result for the EDP in $\Omega$ with vanishing boundary data at $\partial \Omega$ and requiring that the solution diverges to infinity at any point in $\partial_\infty\H^2$ is presented there.

This section is organized as follows: We present some facts about $X-$graphs over $E_H$ in Proposition \ref{graphundergraph} and a way to take the rotational surfaces from Section \ref{sec-rot} to have boundary on $E_H$ in Lemma \ref{lem-radiaisHneq0}. With these tools, Theorem \ref{teoHneq0} can be proved following the same steps as Theorem \ref{teo-H=0}.

	\begin{proposition}\label{graphundergraph}
		Let $ E_H \subset \mathbb{H}^{n+1} $ be an umbilical hypersurface with CMC $ H  $ that is equidistant to  $ \mathbb{H}^n $ and $ X $ a 	hyperbolic Killing field that is orthogonal to $ \mathbb{H}^n.$ Consider  $ w: \mathbb{H}^n \to \mathbb{R} $ the function such that $ {\rm Gr}(w)= E_H.$ Then, given a function $ u:E_H \to \mathbb{R}, $ we have that $ \tilde{u}: \mathbb{H}^n \to \mathbb{R} $ defined by $$ \tilde{u}(\tilde{x})=u(\varphi_{w(\tilde{x})}\tilde{x})+ w(\tilde{x}) $$ is such that $ {\rm Gr}(\tilde{u})={\rm Gr}(u) $ and $$ \grad \tilde{u}(\tilde{x})= \langle \grad u(x), X(x)\rangle_x \grad w(\tilde{x})+\left[D_{2}\varphi_{w(\tilde{x})}(\tilde{x})\right]^{-1}\pi (\grad u(x)),$$ where $\pi :T_xE_H \to T_{x}\left(\varphi_{w(\tilde{x})}\mathbb{H}^n\right)$ is the natural projection and $ x=\varphi_{w(\tilde{x})}(\tilde{x}).$
		
	\end{proposition}

\begin{figure}[h!]
	\centering
	\includegraphics[width=0.7\linewidth]{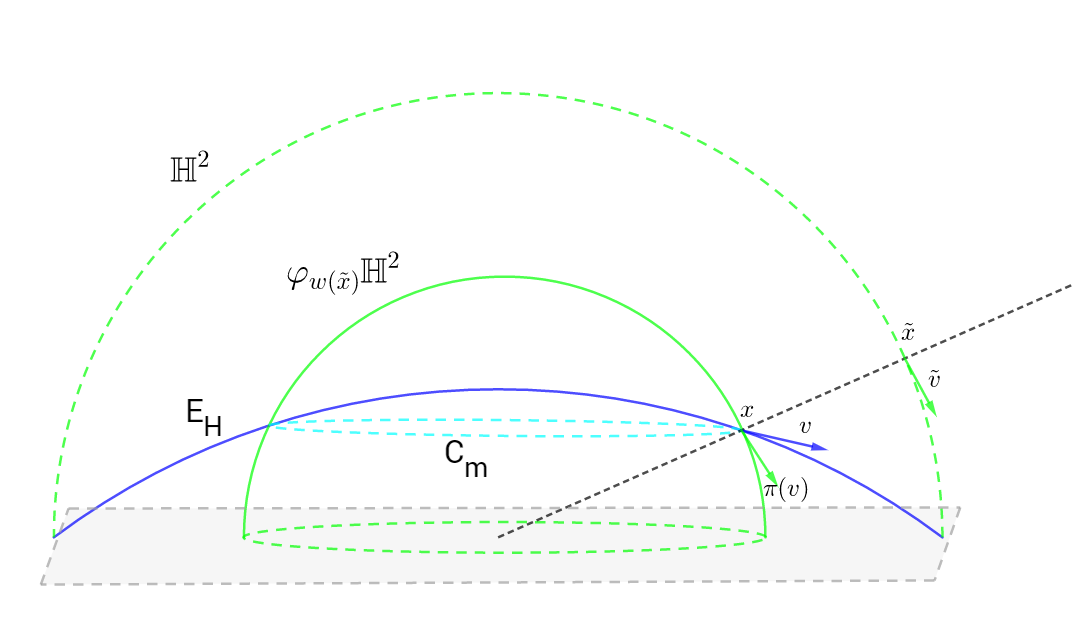}
	\caption{Relating $\grad u$ and $\grad \tilde{u}$}
	\label{fig:proposicaogradiente}
\end{figure}
	
	\begin{proof}
		
Observe that $X-$graphs with domain $\Omega \subset E_H$ are well defined by
$${\rm Gr}(u)=\{\varphi_{u(x)} x\;|\;x \in \Omega \}.$$	
		
		First, we show that $ {\rm Gr}(\tilde{u})={\rm Gr}(u). $ Denoting by $ x=\varphi_{w(\tilde{x})}(\tilde{x}), $ 
$$\begin{array}{ll}
{\rm Gr}\left(\tilde{u}\right)& = \{\varphi_{\tilde{u}(\tilde{x})} \tilde{x}\;|\;\tilde{x}\in \mathbb{H}^n\} \\ \vspace{.1cm}
&= \left\{\varphi_{u(x)+w(\tilde{x})} \tilde{x}\;|\;\tilde{x}\in \mathbb{H}^n\right\}\\ \vspace{.1cm} 
& =\{ \varphi_{u(x)}\circ \varphi_{w(\tilde{x})}\tilde{x}\;|\;\tilde{x}\in \mathbb{H}^n\}\\ \vspace{.1cm} & =\{\varphi_{u(x)} x\;|\;x \in E_H\}=  {\rm Gr}\left(u\right).
\end{array}$$

		For the computation of $ \grad \tilde{u}(\tilde{x}),$ we consider a curve $ \tilde{\alpha}: (-\varepsilon,\varepsilon) \to \mathbb{H}^n $ such that  $ \tilde{\alpha}(0)=\tilde{x} $ and $ \tilde{\alpha}'(0)=\tilde{v}. $ Thus    
		\begin{equation}\label{derivadadou}
		\tilde{u}(\tilde{\alpha}(t))=u(\alpha(t))+w(\tilde{\alpha}(t))
		\end{equation}
		where $\alpha (t)=\varphi_{w(\tilde{\alpha}(t))}(\tilde{\alpha}(t)).$ Using the chain rule we obtain  
$$\begin{array}{ll}
			\alpha'(t)&=(w\circ \tilde{\alpha})'(t)\frac{d}{dt}\varphi_{w(\tilde{\alpha}(t))}\tilde{\alpha}(t) +D_2\varphi_{w(\tilde{\alpha}(t))}\tilde{\alpha}(t)\cdot \tilde{\alpha}'(t) \vspace{.3cm} \\
          &=(w \circ\tilde{\alpha})'(t) X(\alpha(t)) + D_2\varphi_{w(\tilde{\alpha}(t))}\tilde{\alpha}(t)\cdot \tilde{\alpha}'(t) 
		\end{array}$$    

		and for $t=0,$ 
		\begin{equation}\label{aderivada}
		\alpha'(0)=\langle \grad w(\tilde{x}), \tilde{v} \rangle_{\tilde{x}} X(x)+D_2\varphi_{w(\tilde{x})}(\tilde{x})\cdot \tilde{v}.
		\end{equation}    
		Using \eqref{derivadadou} and  \eqref{aderivada} we obtain 
		$$\begin{array}{ll}
			\langle \grad \tilde{u}(\tilde{x}), \tilde{v} \rangle_{\tilde{x}}  &= \langle \grad u(x), \alpha'(0) \rangle_x + \langle \grad w (\tilde{x}) , \tilde{\alpha}'(0) \rangle_{\tilde{x}} \vspace{.3cm} \\
			&= \langle \grad u(x), \langle \grad w(\tilde{x}), \tilde{v} \rangle_{\tilde{x}} X(x)\rangle_x\vspace{.3cm}  \\
			&+ \langle \grad u(x), D_2\varphi_{w(\tilde{x})}(\tilde{x})\cdot \tilde{v} \rangle_x + \langle \grad w(\tilde{x}) , \tilde{v} \rangle_{\tilde{x}}.
		\end{array}$$
		
    Since $D_2\varphi_{w(\tilde{x})}(\tilde{x})\cdot \tilde{v} \in \{X(x)\}^\perp =T_{x}\varphi_{w(\tilde{x})}\H^n \subset T_x\H^{n+1},$ $$\langle \grad u(x), D_2\varphi_{w(\tilde{x})}(\tilde{x})\cdot \tilde{v} \rangle_x=\langle \pi(\grad u(x)), D_2\varphi_{w(\tilde{x})}(\tilde{x})\cdot \tilde{v} \rangle_x,$$  $$\pi: T_x\H^{n+1} \to T_{x}\H^n$$ is the projection  
$$\pi(y)=y-\langle y, X( x)\rangle_x	\frac{X(x)}{|X(x)|^2}.$$ Moreover, since $X$ is a Killing vector field, $D_2\varphi_{w(\tilde{x})}(\tilde{x}): T_{\tilde{x}}{\H^n} \to T_{ \varphi_{w(\tilde{x})}(\tilde{x})}\H^n$ is an isometry and 
		$$\langle \pi(\grad u(x)), D_2\varphi_{w(\tilde{x})}(\tilde{x})\cdot \tilde{v} \rangle_x= \langle [D_2\varphi_{w(\tilde{x})}(\tilde{x})]^{-1}\cdot \pi(\grad u(x)),  \tilde{v} \rangle_x.$$
Hence, 		$$\begin{array}{ll}
           \langle \grad \tilde{u}(\tilde{x}), \tilde{v} \rangle_{\tilde{x}} &=
\langle\langle \grad u(x),X(x) \rangle_x  \grad w(\tilde{x}), \tilde{v} \rangle_{\tilde{x}} \vspace{.3cm} \\
			&+ \langle D_{2}^{-1}(\varphi_{w(\tilde{x})}(\tilde{x}))\pi (\grad u(x)), \tilde{v} \rangle_{\tilde{x}} 
			+ \langle \grad w(\tilde{x}), \tilde{v} \rangle_{\tilde{x}}.
		\end{array}$$
		
	\end{proof}

\begin{remark}\label{rmk-w}
We may compute $w$ following the  steps from Section \ref{sec-rot} adapted 
 and obtain $$w(\tilde{x})= -\int_{r(\tilde{x})}^\infty \frac{H\tanh(t)}{\cosh(t)\sqrt{1-H^2\tanh^2(t)}}dt$$ which is a negative increasing function that goes to zero as $r(\tilde{x})$ goes to infinity.
 
 Besides, $|w(\tilde{o})|$ is an incresing function of $H$ that diverges as $H \to 1.$ 
\end{remark}	
	
 Since $w$ is a bounded $C^\infty$ function, the following holds

	\begin{corollary}\label{cor-uew}
		Under the conditions of the above proposition we have:
		\begin{enumerate}
			\item [(i)] If $ r (\tilde {x}) \to \infty, $ 
			then $ |\grad \tilde{u} (\tilde{x}) | \to 0 $ if and only if $ | \grad u(x) | \to 0 $
			
			\item [(ii)] $ | \grad \tilde {u} (\tilde {x}) | \to + \infty $ if and only if $ | \grad u (x) | \to + \infty. $
		\end{enumerate}
	\end{corollary}
	
The next Lemma 	presents the family of functions $f_R$ whose graphs are $X-$trans\-lations of the graphs of $v_R.$

	\begin{lemma}\label{lem-radiaisHneq0}
	Given $R>0,$ let $D_R$ be the ball of radius $R$ centered at the origin of $E_H.$ Then there is a function $f_R : E_H \backslash D_R \to \R$ satisfying:
\begin{enumerate}
\item[(i)] $f_R$ vanishes on $\partial D_R;$
\item[(ii)] $f_R$ has CMC $H$ $X-$graph which is tangent to the $X-$Killing cyllinder that passes through $\partial D_R;$
\item[(iii)] $f_R$ is bounded by $B(H);$
\item[(iv)]\label{it4} The gradient of $f_R$ at $x$ goes to zero, as $x$ goes to infinity. 
\item[(v)] $f_R$ depends only on the distance to the origin of $E_H.$
\end{enumerate}
\end{lemma}

\begin{proof}

Consider $\H^n$ the totally geodesic hypersurface that has the same asymptotic boundary as $E_H$ and let $w:\H^n \to \R$ be the function from Remark \ref{rmk-w} such that $E_H$ is the $X-$graph of $w$ (See Figure \ref{fig:proposicaogradiente}). We actually need the `inverse' of $w,$ i. e.,  let $\hat{w}: E_H \to \R$ be the function such that $\H^n$ is the $X-$graph of $\hat{w}.$ It holds that $$\hat{w}(x)=-w(\tilde{x}) \text{ for }x=\varphi_{w(\tilde{x})}(\tilde{x}).$$ Since $w$ ($\hat{w}$) is radial, we may define $w(\tilde{R})=w(\tilde{x})$ ($\hat{w}(R)=\hat{w}(x)$) for any $\tilde{x}\in \H^n$ ($x\in E_H$).

	Observe that $u: U \subset E_H \to \R$ has CMC $H$ $X-$graph if and only if 
	$$\begin{array}{l}
	\tilde{u}:  \tilde{U}=\{\varphi_{\hat{w}(x)}(x)\,|\, x\in U \}\subset \H^n \to \R\vspace{.3cm} \\
	 \tilde{u}(\tilde{x}) = u(\varphi_{w(\tilde{x})}(\tilde{x}))+w(\tilde{x})	
	\end{array}$$ satisfies $\mathcal{M}_{H}(\tilde{u})=0$ for $\mathcal{M}_{H}$ defined in \eqref{operador}. Besides, given a ball $D_R(o) \subset E_H,$ there is an associated ball $D_{\tilde{R}}(\tilde{o}) \subset \H^n$ which is the intersection of the $X-$Killing cyllinder though $D_R(o)$ with $\H^n.$ 
	
	For any $R>0,$ consider the function $v_{\tilde{R}, H}: \H^n\backslash D_{\tilde{R}}\to \R$ from Proposition \ref{rad-sim} and define $f_R: E_H \backslash D_R \to \R$ by $$f_R(x)= 
	 v_{\tilde{R}, H}(\varphi_{\hat{w}(x)}(x))-[\hat{w}(R)-\hat{w}(x)].$$
	
The graph of $f_R$ has CMC $H$ because it is the graph of $$x\in \H^n\backslash D_R \mapsto v_{\tilde{R}, H}(x)+w(\tilde{R}).$$ Its boundary is $\partial D_R \subset E_H$ and it is tangent to the Killing cyllinder $\varphi_{\R}(D_R)$ because of the second item in Corollary \ref{cor-uew}. Besides, $f_R$ depends only on the distance to $o \in E_H$ since it is the difference between two other radial functions: $v_1=v_{\tilde{R}, H}(\varphi_{\hat{w}(\cdot)}(\cdot))$ and $v_2=\hat{w}(R)-\hat{w}(\cdot).$

In order to see that $|f_R| \leq B(H),$ it is sufficient to prove that $$0\leq v_1 \leq B(H) \text{ and }0\leq v_2 \leq B(H).$$

The first inequality is demonstrated in Lemma \ref{lem-radiais}. For the second one, notice that $r \mapsto \hat{w}(r)$ is a positive decreasing function. Therefore,
$\hat{w}(R)-\hat{w}(r)\leq \hat{w} (o)$
and it remains to show that
$\hat{w} (o) \leq B(H),$ which reads as
\begin{eqnarray*}
\int_{0}^\infty \frac{H\tanh(t)}{\cosh(t)\sqrt{1-H^2\tanh^2(t)}}dt\\
\leq \int_{0}^{+\infty}\frac{H+(1-H)e^{-2t}}{\cosh(t)\sqrt{1-(H+(1-H)e^{-2t})^2}}dt.\end{eqnarray*}
We may compare the integrands and see that this inequality holds. For that, since 
 $ x\mapsto \frac{x}{\sqrt{1-x^2}} $ is an increasing function, we only notice that $$H\tanh(t) \leq H+(1-H)e^{-2t} \text{ for }t\geq 0.$$

Item (iv) 
 is a consequence of the first claim in Corollary \ref{cor-uew}.
\end{proof}

The following theorem has the precise statement for Theorem \ref{teoHneq0-intro}.	
	
	\begin{theorem}\label{teoHneq0}
		Let $ E_H \subset \mathbb{H}^{n+1} $ be an umbilical hypersurface with CMC $ H \in [0,1).$ Consider $ \Omega \subset E_H $ an exterior domain of class $ C^0 $ and $ \gamma $ an oriented geodesic passing orthogonally through $ E_H, $ where $ X $ is the hyperbolic Killing field tangent to $ \gamma $ satisfying $ \langle X, \eta \rangle > 0, $ where $ \eta $ is the normal vector to $ E_H. $ Then, given a real number $ s \geq 0 $ there is a nonnegative bounded  function $ u  \in C^{ \infty} (\Omega) \cap C^0 (\overline{\Omega}) $ such that the $ {\rm Gr}(u) $ has constant mean curvature $ H, $ $ u |_{\partial \Omega} = 0,$  $ \displaystyle \limsup_{x \rightarrow \partial \Omega } | \grad u(x) | = s$ and $\sup_{\Omega} u\leq B(H).$
	\end{theorem}
	
	\begin{proof}
	We start fixing $ s \geq 0 $ and assuming $ \Omega \subset E_H $ is a $ C^{\infty} $ exterior domain. Let $D_{0},...,D_{m}, ...$ be a collection of balls in $ E_{H}$ centered at the origin $o = \gamma \cap E_{H},$ with radius $ R_m, $ satisfying the condition that $ D_0 $ is the smallest ball such that $E_H\backslash \Omega \subset D_{0},$ $$D_{i} \subset D_{i+1},\,\,  E_H =\displaystyle\cup_{i \in  \mathbb{N}}D_{i}$$ and $R_1$ is large enough so that \eqref{eq-condR1} (see below) holds.

As in Theorem \ref{teo-H=0}, let  $ \Omega_{m} = D_{m} \cap \Omega,$ set $\Gamma_{m}= \partial D_{m}$  and consider the set
\begin{equation*} T_{m} = \{ t \geq 0 \;|\;\exists u_{t} \in C^{\infty}\left(\overline{\Omega_{m}}\right) \text{solution of PD$^H_m(t)$}\},
\end{equation*}
where $PD^H_m(t)$ has the same boundary conditions as $PD_m(t)$ and the PDE corresponds to ask that $u_t$ has CMC $H$ $X-$graph.

Next we prove the following items: 
		\begin{enumerate} 
			\item[1)] $T_{m}$ is not empty
			\item[2)] $T_{m}$ is bounded
			\item[3)] $h(m)=\sup{T_{m}}$ belongs to $T_{m}$
			\item[4)]  $\displaystyle \limsup_{x \to \partial \Omega} | \grad u_{h(m)}| = s$.
		\end{enumerate}
Afterwards we will prove that $u_m=u_{h(m)} \in C^{\infty}(\Omega_m)$ has CMC $H$ $X-$graph and $\sup |\grad u_m | = s.$ In the end, the fact that $\Omega=\cup_{m}\Omega_m$ will allow us to find the solution $u_s.$

The first claim follows from taking $u=0,$ since $E_H$ has CMC $H.$ Therefore $t=0$ belongs to $T_m.$

For the remaining items we use the functions $f_R,$ which play the same role as the functions $v_R$ did in the proof of Theorem \ref{teo-H=0}. They were presented in Lemma \ref{lem-radiaisHneq0}.

For the second item, consider the family of graphs of $\{f_R\}_{R\leq R_m}.$ If $j>B(H)$ belongs to $T_m,$ the graph of the associated $u_j$ would have a first interior contact point with a member of this family, leading us to a contradiction. Hence $T_m$ is bounded by $B(H).$

In order to find the function $u_m=u_{h(m)}$, we prove next that $h(m) \in T_m.$ First, we exhibit the condition on $R_1:$ 
\begin{equation}\label{eq-condR1}
|\grad f_{R_0}(x)|_{\partial D_R}\leq s/2 \, \text{for} \, R\geq R_1,
\end{equation} 
which is possible since the gradient of $f_R$ at $x$ goes to zero, as $x$ goes to infinity. Observe that our choice of $R_1$ guarantees that, for all $m,$ if $t\in T_m,$ then $$|\grad u_t|\leq s/2 \text{ on } \Gamma_m.$$ This is again a consequence of the maximum principle as in the proof of Theorem \ref{teo-H=0}. The functions constant $t,$ $u_t$ and $f_{R_0}-(f_{R_0}(R_m)+t)$ satisfy $$f_{R_0}-(f_{R_0}(R_m)+t)\le u_t\le t \text{ in }\Omega_m.$$ Hence $$|\grad u_t|\leq |\grad f_{R_0}(R_m)|<s/2 \text{ on }\Gamma_m.$$ Therefore for $t\in T_m,$ $\sup_{\partial\Omega_m} |\grad u_{t}|\leq s.$ %
		
We conclude that $h(m)\in T_m$ following the same ideas as in the proof of Theorem \ref{teo-H=0}. We remark that the boundedness for $\grad u_t$ obtained using Lemma \ref{lem-gradest} also holds here since Proposition \ref{graphundergraph} gives a nice bijection from functions defined in $E_H$ to functions defined in $\mathbb{H}^n.$

The fourth property of $T_m,$  $$\displaystyle \limsup_{x \to \partial \Omega} | \grad u_{h(m)}| = s,$$ also holds by the argument presented in Theorem \ref{teo-H=0}. 

Finally, we consider the sequence $u_m=u_{h(m)}$ and conclude that, up to a subsequece, it converges to a function $u_s$ satisfying the statement.

The general case of $C^0$ domains also holds by the same argument as in Theorem \ref{teo-H=0}.

	\end{proof}

	\bigskip
	
Patricia Klaser: Universidade Federal de Santa Maria, 

 patricia.klaser@ufsm.br.

Adilson Nunes: Universidade Federal do Rio Grande,

 adilson.nunes@furg.br.

Jaime Ripoll: Universidade Federal do Rio Grande do Sul,

 jaime.ripoll@ufrgs.br

\end{document}